\documentclass[titlepage,12pt]{article}

\usepackage{array}
\usepackage[reqno]{amsmath}
\usepackage{xcolor}

\usepackage{amsthm}

\theoremstyle{plain}
\newtheorem{proclaim}{PROCLAIM}[section]
\newtheorem{theorem}[proclaim]{Theorem}
\newtheorem{lemma}[proclaim]{Lemma}

\newtheorem{corollary}[proclaim]{Corollary}

\theoremstyle{definition}
\newtheorem{definition}[proclaim]{Definition}

\newtheorem{assumption}[proclaim]{Assumption} 

\usepackage{graphicx}
\usepackage{hyperref}  
\usepackage{booktabs}
\usepackage{tikz}
\usetikzlibrary{shapes.geometric, arrows.meta, positioning}

\usepackage{mathtools}
\usepackage{enumitem}

\usepackage[
	paper=letterpaper,
	margin=1in,
	headheight=13.6pt,
	includehead,
	includefoot
]{geometry}

\newcommand{\cE}{{\cal E}}
\newcommand{\reals}{{I\kern-.35em R}}
\newcommand{\Reals}{\overline{I\kern-.35em R}}
\newcommand{\lset}{\big\lbrace}
\newcommand{\mset}{{\,\big\vert\,}}
\newcommand{\rset}{\big\rbrace}
\newcommand{\Lset}{\Big\lbrace}

\newcommand{\Rset}{\Big\rbrace}

\newcommand{\Ex}{{I\kern-.35em E}}
\newcommand{\bfxi}{\mbox{\boldmath $\xi$}}
\newcommand{\ball}{{I\kern -.35em B}}
\def\state #1. { \noindent{\bf#1.\enspace}}
\newcommand{\eop}{\hfill{$\vcenter{\hrule height1pt \hbox{\vrule width1pt height5pt\kern5pt \vrule width1pt} \hrule height1pt}$} \medskip}

\newcommand{\upto}{{\raise 1pt \hbox{$\scriptstyle \,\nearrow\,$}}}
\newcommand{\cB}{{\cal B}}

\newcommand{\nargmaxinf}{\mathop{\rm argmaxinf}\nolimits}
\newcommand{\nats}{{I\kern -.35em N}}
\newcommand{\lopto}{\,{\lower 1pt\hbox{$\longrightarrow$}}\kern -16pt
	\hbox{\raise 4pt \hbox{$\,\scriptstyle lop$}}\hskip6pt}

\newcommand{\argmax}{\mathop{\rm argmax}}
\newcommand{\argmin}{\mathop{\rm argmin}}
\newcommand{\dom}{\operatorname{dom}}

\begin{document}

		\begin{titlepage}
			\vglue 0.5cm
			\begin{center}
				\begin{large}
					{\bf Solving equilibrium problems in economies with financial markets, home production, and retention}
					\smallskip
				\end{large}
				\vglue 1.truecm
				\begin{tabular}{c}
					\begin{large} {\sl Julio Deride
					} \end{large}
					\\
					Faculty of Engineering and Sciences\\
					Universidad Adolfo Ibáñez\\
					julio.deride@uai.cl
				\end{tabular}
			\end{center}
			\vskip 0.25truecm
			\noindent
			{\bf Abstract}.\quad We propose a new methodology to compute equilibria for
			general equilibrium problems in exchange economies with real financial markets,
			home production, and retention. We demonstrate that equilibrium prices can be
			determined by solving a related maxinf-optimization problem.
			We incorporate the non-arbitrage condition for financial markets
			into the equilibrium formulation and establish the equivalence between solutions
			to both problems. This reduces the complexity of the original problem by eliminating the
			need to directly compute financial contract prices, allowing us to calculate
			equilibria even in cases of incomplete financial markets.
			
			We also introduce a \emph{Walrasian} bifunction that
			captures market imbalances and show that the maxinf-points of this function correspond to
			equilibrium points. Moreover, we demonstrate that every equilibrium point can be
			approximated by a limit of maxinf-points for a family of perturbed problems, relying
			on the notion of \emph{lopsided convergence}.
			
			Finally, we propose an \emph{augmented Walrasian} algorithm and present
			numerical examples to illustrate the effectiveness of this approach. Our
			methodology allows for the efficient calculation of equilibria in a variety of exchange
			economies and has potential applications in finance and economics.	
			\vskip 0.25truecm
			\halign{&\vtop{\parindent=0pt
					\hangindent2.5em\strut#\strut}\cr
				{\bf Keywords}: Walras equilibrium, stochastic equilibrium,
				lopsided convergence, general equilibrium theory, augmented Walrasian, incomplete markets, GEI, no-arbitrage.
				\hfill\break\cr
				{\bf JEL Classification}: C680, D580, C620 \cr\cr
				{\bf Date}: \today \cr}
		\end{titlepage}

\section{Introduction}\label{geiintro}
		
		In this paper, we investigate the computational challenges of solving general equilibrium problems for economies with real asset markets, home production, and retention, including the case when financial markets are incomplete. We follow the model in \cite{JoRoWe14gei}, where economic equilibrium for incomplete financial markets in general real assets is developed in a new formulation with currency-denominated prices. This model is an extension of the classic Walrasian economic equilibrium model presented by \cite{debreu1959theory,ArDe54}, where agents optimize individually under perfect competition, and markets clear. In particular, we consider an inter-temporal economy, where agents face uncertainty over the future and can make financial decisions by accessing a collection of real assets. Each agent also has access to domestic technologies to transform goods between periods and, in addition, has the possibility of deciding to retain goods. In this setting, we focus our attention on the design of a methodology for finding equilibrium prices, following an optimization approach to the equilibrium problem.
		
		Solving general equilibrium problems requires numerical methods to solve optimization problems that, in more realistic models, contain utility functions or a financial market structure leading to computational difficulties, such as nonconvexities, nonsmoothness, and unboundedness.
		The classic solution strategies rely heavily on stringent assumptions,
		such as interior initial endowments, differentiability, and constraint qualifications,
		in order to give a characterization of the equilibrium points using a differential topology
		framework. We propose an algorithm that works under weaker assumptions and that
		follows a different approach from classical methods, relying on sequentially solving
		optimization problems.
		
		The first computational methods presented to solve pure exchange
		economy equilibrium problems were constructed by using a characterization
		of equilibrium prices as the solutions to a fixed-point problem. The most
		relevant solution strategies can be summarized in four main areas: \emph{Simplicial
			methods}, proposed by Scarf \cite{ScaHan73}, based on a fixed-point description
		of the problem, followed by an enumerative argument exploiting the
		geometry of the simplex of prices; \emph{Global Newton methods},
		where the equilibrium problem is presented as a dynamical system for the
		price vector, and it is solved by Newton's iterations; \emph{Homotopy methods},
		where the equilibrium state is seen as the solution of an extended system of
		nonlinear equations describing optimality and market clearing conditions. These methods are the most popular, and
		the solution strategy in this case is based on homotopy methods that solve
		iteratively a sequence of simpler problems that are homotopic to the original one
		and are especially developed for models that incorporate incomplete financial markets
		(see \cite{Sai83homo} for homotopy methods; \cite{Gea90gei,won2019new} for GEI and pre-GEI models; \cite{zhan2021smooth,zhan2021determination} for a smooth homotopy for GEI;
		\cite{Eav11homo,BrDeEa96gei,DeMEav96gei,Ea99,HeSch06,HeKu02,Ma15compgei}
		for computational approaches; and \cite{garg2015complementary} for a
		Lemke's scheme with separable piecewise linear concave utility functions).
		Finally, \emph{Optimization methods} represent equilibrium prices
		as the (direct) solution points of an optimization problem, solved by an interior point
		method \cite{BrKu08} or by an augmented Lagrangian method \cite{Sch13computing}
		(some references and details are omitted for simplicity of exposition).
		Also, a family of methods based on tools from Variational Analysis and generalized
		equations has been proposed, for example, solving the equilibrium problem via a
		virtual market coordinator and epigraphical convergence \cite{borges2021decomposition},
		a generalized Nash game (GNEP) for risk-averse stochastic equilibrium problems
		\cite{luna2016approximation}, and the stochastic variational scheme presented in \cite{MILASI2021125243,DONATO20181353}.
		
		We propose a solution strategy that is based on the correspondence between
		equilibrium prices and the (local) saddle points of the Walrasian bifunction.
		This idea extends our previous work on equilibrium problems \cite{DeJoWe16awlrs},
		where our approach exhibits an alternative with theoretical advantages in terms of stability and
		satisfactory performance in numerical analysis. We now incorporate a broader class
		of economic models, where the inclusion of financial markets,
		home production, and retention is considered. Moreover, we provide a
		convergence proof for our algorithm based on the notion of \emph{lopsided
			convergence} for bifunctions and design a version of our augmented
		Walrasian algorithm for this class of economies.
		
		This article is organized as follows: in Section~\ref{sec:agprob}, we provide
		a mathematical description of the equilibrium model with financial markets,
		home production, and retention. In particular, we describe the agent's
		optimization problem and characterize the financial markets.
		In Section~\ref{sec:geieq}, we study equilibrium conditions and the
		formulation of the problem of the maxinf optimization type. Additionally,
		we present our approximation approach, along with our solution method.
		Finally, Section~\ref{geinum} illustrates our algorithm's applicability,
		including an example of its computational implementation and our results.
		
		\section{General equilibrium with financial markets model}\label{sec:agprob}
		
		The model described in this section corresponds to an extension of the
		general equilibrium model with real financial markets and retainability described
		in \cite{JoRoWe14gei}, where we incorporate home production. This addition
		enriches the model without adding technical complications. Hence,
		the original model, in its form, could have included it by simply expanding the notation.
		Essentially, we consider a two-stage pure exchange economy where agents
		decide on consumption and financial assets, and where second-stage outcomes are uncertain.
		
		We denote the set of real numbers by $\reals$, the set of extended real numbers by $\Reals=\reals\cup\{-\infty,+\infty\}$, and the $m$-dimensional Euclidean space by $\reals^m$. We denote the ball centered at $\bar{x}\in\reals^m$ with radius $r>0$ by $\ball_m(\bar{x},r)=\lset x\in\reals^m\,\mset\,|x-\bar{x}| < r\rset$. We denote by $\ball_m=\ball_m(0,1)$. Recall the definition of the inner product as $\langle x,y\rangle =\sum_{i=1}^m x_iy_i$, for $x,y\in\reals^m$. For $V\subset \reals^m$, we define the linear span of $V$ as the set of all linear combinations of the vectors in $V$, denoted by $\langle V\rangle$.
		
		We say that a function $f:\reals^m\to\Reals$ is upper semi-continuous (usc) at a point $\bar{x}\in\reals^m$ if $\limsup_{x\to\bar{x}} f(x)\leq f(\bar{x})$. If $f$ is usc at every point in $\reals^m$, we say that it is usc. We also define the domain of $f$ as $\dom(f)=\lset x\in\reals^m\,\mset\,f(x)<+\infty\rset$. Additionally, we define the set $\argmax f=\lset x^\prime\in\dom f\,\mset\,f(x^\prime)\geq \sup f\rset$.

		\subsection{Economy description}
		Consider an economy $\cE$ with a finite number of heterogeneous agents,
		indexed by $i=1,\ldots,I$. They trade goods in a market operating in two
		stages, where the present or first stage is denoted by
		$\xi=0$, and there is uncertainty about the outcomes in the second stage,
		where a finite number of possible scenarios is considered, indexed by $\xi=1,\ldots,\Xi$.
		Goods or commodities, as in \cite{JoRoWe14gei}, are \emph{generalized goods}\footnote{They
			can be anything tradable in markets that enters each state of the economy in fixed
			supply within the agents' holdings \cite[p.311]{JoRoWe14gei}.} and can be
		perishable or durable. We consider a finite number of commodities, indexed
		by $l=0,\ldots,L$. Thus, the space of commodities is given by $\reals_+^{(1+L)\times(1+\Xi)}$.
		
		Agents $i=1,\ldots,I$ are endowed with an initial bundle of goods
		$e^i=(e^i_0,e^i_1,\ldots,e^i_\Xi)$, and they trade their holdings,
		optimizing their preferences over goods. These preferences are modeled by a utility function that takes into consideration the inter-temporal decision under uncertainty
		regarding the possibility of consumption and retention. Retainability of a particular
		good reflects the desirability of buying it for possible later enjoyment, and retention can
		occur in both stages. Thus, the utility function that captures these
		preferences is represented by a function of two variables: consumption, denoted
		by $c^i=(c^i_0,c^i_1,\ldots,c^i_\Xi)$, and retention, $w^i=(w^i_0,w^i_1,\ldots,w^i_\Xi)$.
		Moreover, agents have at their disposal the bundle of goods
		that they decided to retain in the first stage (if any), where its availability depends
		on the second-stage scenario and is modeled by matrix multiplication.
		Specifically, if an agent retains $w^i_0$ at the first stage, then they have
		an amount $A_\xi w^i_0$ available in scenario $\xi$ of the second stage.
		An example of this is wine consumption: it can be
		consumed in the first period or retained for the next. In a desirable outcome, a lucky consumer gets better wine; in an undesirable outcome, it can turn into vinegar. Another example is art: buying a painting from an unknown artist can be enjoyed for its own merits, but if that artist becomes the next Michelangelo, the retained painting becomes more valuable.
		
		Denote this utility function by $U^i:\reals_+^{(1+L)\times(1+\Xi)}\times\reals_+^{(1+L)\times(1+\Xi)}\to\Reals$,
		and assume it is usc, with a \emph{survival set} given by $C^i=\dom U^i$. For example,
		we consider agents that are maximizing their expected utility, i.e., their objective
		function takes the following form
		\[U^i(c,w)=u^i(c_0,w_0)+\Ex\{u^i(c_{\bfxi},w_{\bfxi})\},\]
		for $u^i:\reals_+^{(1+L)}\times\reals_+^{(1+L)}\to\Reals$.
		
		In this economy, agents have access to \emph{home production} technology,
		described as a process consisting of a collection of activities
		that requires a bundle of goods as input and produces a (stochastic) bundle
		of output goods in the second stage. This procedure is modeled by a
		decision variable $y$ belonging to a set of $m$-acceptable activities, and a
		pair of input/output matrices $(T^i_0,\{T^i_{\bfxi}\})$ at the disposal of agent $i$.
		Thus, by using $T^i_0 y$ resources, a random output $T^i_{\bfxi} y$
		is obtained in the second stage. We assume that the set of acceptable activities, $Y^i\subset\reals^m_+$, is a closed convex cone, and it includes
		the possibility of no-action, i.e., $0\in Y^i$. One possible interpretation
		for the matrices determining input/output (home production)
		is to consider it as a process that could represent savings (including
		enhancements or deterioration), investment activities, retention, and so on.
		
		The financial market structure considered in this economy
		is detailed as follows. Agents have access to a finite number of real contracts,
		indexed by $j=0,\ldots,J$, where each contract has an
		associated return vector of goods, depending on the second-stage outcome. To fix notation, for each unit of the $j$-th marketed asset, the stochastic return delivered is described by a bundle of goods denoted $D_{j,\bfxi}$.
		We are not imposing any explicit restrictions on financial transactions happening
		simultaneously. However, additional provisions for each transaction
		(real transaction costs) are modeled by an associated cost vector of
		goods, $D_0$, for every unit of a contract that is issued (sold).
		
		Under this market characterization, and without providing any further
		description of the uncertainty, the market information structure is
		organized as a tree, as depicted in Figure~\ref{fig:geistr}.
		
		\begin{figure}[!hbt]
			\begin{center}
				\begin{tikzpicture}[
					node distance=2.5cm and 1.5cm,
					state/.style={
						rectangle,
						draw=black!20,
						fill=black!5,
						text width=3.2cm,
						align=center,
						minimum height=1.25cm,
						font=\scriptsize
					},
					edge_style/.style={-Stealth, draw=black!70, thick}
					]
					
					\node[state] (xi0) {
						$\xi=0$ \\
						$(c_{0},w_{0},y,z)$ \\
						$-D_{0}z,-T_{0}y$
					};
					
					\node[state, above right=0.2cm and 2cm of xi0] (xi1) {
						$\xi=1$ \\
						$(c_{1},w_{1})$ \\
						$D_{1}z,A_{1}w_{0},T_{1}y$
					};
					
					\node[state, right=0.2cm and 2cm of xi0] (xi2) {
						$\xi=2$ \\
						$(c_{2},w_{2})$ \\
						$D_{2}z,A_{2}w_{0},T_{2}y$
					};
					
					\node[font=\huge, below=0.1cm of xi2, text=black!60] (dots) {\vdots};
					
					\node[state, below right=1.2cm and 2cm of xi0] (xiXi) { 
						$\xi=\Xi$ \\
						$(c_{\Xi},w_{\Xi})$ \\
						$D_{\Xi}z, A_{\Xi}w_{0}, T_{\Xi}y$
					};
					
					\draw[edge_style] (xi0) -- (xi1);
					\draw[edge_style] (xi0) -- (xi2);
					\draw[edge_style] (xi0) -- (xiXi); 
					
				\end{tikzpicture}
				\caption{\small Market structure of the GEI model. First-stage decision variables for consumption, retention, home production, and financial portfolio are summarized by $(c_0,w_0,y,z)$, respectively, along with the associated transaction costs of $-D_0z$ and home production input $-T_0y$. In the second stage (for each scenario $\xi$), consumption and retention variables are $(c_\xi,w_\xi)$; financial, retainability, and production outcomes are represented by $(D_\xi z,A_\xi w_0, T_\xi y)$.}
				\label{fig:geistr}
			\end{center}
		\end{figure}
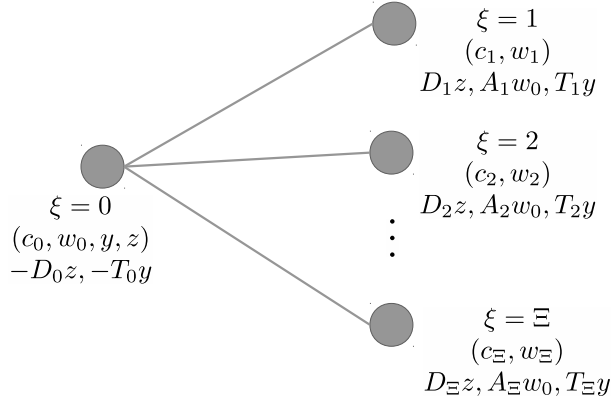
		
		Prices are the final element needed to fully describe the economy.
		Goods and contracts are traded in a perfectly competitive market setting.
		We denote the market prices of goods by
		$p=(p_0,p_1,\ldots,p_\Xi)\in\reals_+^{(1+L)\times(1+\Xi)}$,
		and the market prices of financial contracts by
		$q=(q_0,q_1,\ldots,q_J)\in\reals_+^{1+J}$.
		Hence, under perfect competition, agents are price-takers. Given
		prices $p$ and $q$, each agent $i$ chooses an affordable bundle of
		consumption goods $c^i$, a retention strategy $w^i$, and a financial decision
		(given by a portfolio strategy $z^i=z^{i,+}-z^{i,-}$) that maximizes their utility
		function. Thus, each agent solves the following optimization problem:
		\begin{align*}
			\max_{\{c,w,y,z^+,z^-\}}&\;U^i\left(c,w\right),\\
			{\rm s.t.}&\;\langle p_0,c_0+w_0+D_0z^-+T^i_0y\rangle+\langle q, z^+\rangle \leq \langle p_0,e^i_0\rangle+\langle q,z^-\rangle,\\
			&\langle p_\xi,c_\xi+w_\xi+D_\xi z^-\rangle \leq \langle p_\xi,e^i_\xi+D_\xi z^++A^i_\xi w_0+T^i_\xi y\rangle,\quad\xi=1,\ldots,\Xi,\\
			&(c,w)\in C^i,\quad z^+,z^-\in\reals^{1+J}_+,\quad y\in Y^i.
		\end{align*}
		
		Let $(c^i,w^i,y^i,z^{+,i},z^{-,i})$ be a solution to this problem (which depends on $p$
		and $q$). To simplify notation, we define the \emph{individual excess supply} as
		\begin{equation}\label{sdef}
			\begin{cases}
				s^i_0&=e^i_0-c^i_0-w^i_0-D_0z^{-,i}-T^i_0y^i,\\
				s^i_\xi&=e^i_\xi+D_\xi(z^{i,+}-z^{i,-})+A^i_\xi w^i_0+T^i_\xi y^i-c^i_\xi-w^i_\xi,\quad\xi=1,\ldots,\Xi.
			\end{cases}
		\end{equation}
		
		For this economy, we define the \emph{excess supply function} as
		the difference between aggregate supply and aggregate demand.
		In our case, it is defined as follows:
		\[ {\rm ES}_\xi=\sum_{i=1}^I s^i_\xi,\quad\xi=0,\ldots,\Xi.\]
		
		Given the excess supply functions $\{{\rm ES}_\xi:\xi=0,\ldots,\Xi\}$, we define an \emph{equilibrium}
		for the economy as a price vector $(p,q)$ such that this family of functions is non-negative:
		\[{\rm ES}_\xi \geq 0,\quad\xi=0,\ldots,\Xi.\]
		That is, at equilibrium, the market can satisfy the aggregate demand
		for goods with the available aggregate supply. Furthermore, we
		require that at this price vector, the financial contract market clears:
		\[\sum_i z^{+,i}=\sum_i z^{-,i},\]
		where the total amount of financial contracts issued equals
		the number of contracts purchased, or aggregate short
		positions equal aggregate long positions.
		
		Before we proceed with further details of the model, we state
		one assumption regarding the \emph{desirability} of the available
		goods in the economy, and another assumption regarding
		\emph{survivability} of the agents.
		
		\begin{assumption}{\rm (Indispensability)}\label{ass:indisp}
			There is a good that is indispensable to all agents. Every good
			is indispensable to at least one agent.
		\end{assumption}
		This assumption is weaker than the classic strong monotonicity
		assumption in \cite{MQ2002}, as it only reinforces the indispensability
		of the same good for every agent and guarantees that the goods traded
		in the economy are actually desired by at least one agent.
		Denote the indispensable good for all agents by the first indexed good ($l=0$).
		We will use this good as the \emph{numéraire} for this economy in
		every scenario. Therefore, assume that $p_{\xi,0}=1$ for $\xi=0,\ldots,\Xi$
		(a different approach can be followed by imposing that
		$p_\xi\in\Delta_{L+1}$ for $\xi=0,\ldots,\Xi$\footnote{$\Delta_m$ corresponds to the $m$-dimensional canonical simplex.}). Another advantage of this
		approach is that it helps us to avoid indeterminacy of the price system \cite{patinkin1949indeterminacy}
		and, therefore, eases the comparison between prices for different stages.
		
		Additionally, the stochastic nature of the agents' problems necessitates assumptions regarding the survivability condition or market participation. In line with the approach in \cite{JoRoWe14gei}, we adopt the following assumption:
		
		\begin{assumption}{\rm (Ample survivability assumption)}\label{ass:ample}
			For every agent $i$ and for all possible second-stage scenarios $\xi=1,\ldots,\Xi$, there exist
			a consumption bundle $(\tilde c^i,\tilde w^i)\in C^i$ and an activity level
			$\tilde y\in Y^i$ such that
			\begin{align*}
				e^i_0-\tilde{c}^i_0-\tilde{w}^i_0-T^i_0\tilde{y}&\geq 0,\\
				e^i_\xi-\tilde{c}^i_\xi-\tilde{w}^i_\xi+T^i_\xi\tilde{y}+A_\xi\tilde{w}^i_0&\geq 0,\quad\xi=1,\ldots,\Xi.
			\end{align*}
		\end{assumption}
		Mathematically, these conditions correspond to the relatively complete recourse property, and they guarantee the feasibility of each agent's optimization problem. From the economic perspective, this assumption provides
		a minimum requirement for each agent to be able to engage in trading
		and participate in the economy, independently of the market prices.

		\subsection{Financial market and no-Arbitrage condition}
		We base our equilibrium approach on prices that satisfy the
		\emph{no-arbitrage condition} to rule out the existence of
		arbitrage opportunities in the financial markets (see, for example,
		\cite[Ch.\,5,\S3.3]{Kir98ge}). Simply put, this condition
		relates the prices of financial contracts and their associated returns.
		A trading strategy that delivers a positive payoff without
		requiring any net payment is called \emph{arbitrage}. Formally, the no-arbitrage
		condition is stated as
		\[\mbox{ there is no }z\in\reals^{1+J}\mbox{ such that } \langle q,z\rangle < 0, \text{ and }
		M_1(p)z:=\left[\begin{array}{c} p_1^\top D_1\\ \vdots\\ p_\Xi^\top D_\Xi \end{array}\right]z>0.\]
		The matrix $M_1(p)$ defines the span of payoffs in the financial market at goods prices $p$.
		By defining $M(p,q)=\left[\begin{array}{c} -q^\top \\M_1(p)\end{array}\right]$,
		the no-arbitrage condition is interpreted as the non-existence of a portfolio $z$ such that $M(p,q)z>0$. Equivalently, following an approach similar to that in \cite[Thm.\,9.3]{MQ2002}, this condition can be restated as
		\[\langle M(p,q) \rangle \cap \Delta_{\Xi+1} = \emptyset.\]
		
		Since $\Delta_{\Xi+1}$ is a compact set and $\langle M(p,q) \rangle$ is a closed subspace, there exists
		a strictly separating hyperplane defined by a normal vector $\sigma\neq 0$ such that
		\[\sup_{\tau \in \langle M(p,q) \rangle} \sigma\cdot\tau < \inf_{\tau \in \Delta_{\Xi+1}} \sigma\cdot\tau.\]
		Moreover, $\sigma\in\reals_{++}^{1+\Xi}$ and $\sigma\in\langle M(p,q)\rangle^{\perp}$, i.e.,
		$\sigma_\xi>0$ for $\xi=0,\ldots,\Xi$ and $\sigma\cdot M(p,q)z=0$ for all $z$.
		Finally, (redefining $\sigma\leftarrow\frac{\sigma}{\sigma_0}$ if necessary) the no-arbitrage
		condition can be stated as the existence of positive scalars $\sigma_\xi>0$ for $\xi=1,\ldots,\Xi$ such that
		\begin{equation}\label{nacq}
			q=\sum_{\xi=1}^\Xi \sigma_\xi D_\xi^\top p_\xi.
		\end{equation}
		Therefore, any prices $p$ and $q$ that satisfy the no-arbitrage
		condition have (at least one) associated vector $\sigma$ such
		that Equation~\eqref{nacq} holds. The vector $\sigma$ is usually
		denoted as the vector of \emph{state prices} or a vector that
		\emph{rationalizes} the prices $p$ and $q$. In what follows, we only
		consider prices that satisfy this condition.
		
		\subsection{A solution strategy for the agents' problems}
		We use the no-arbitrage condition to propose an alternative solution strategy
		for the equilibrium problem for this economy. Equilibrium conditions will be described
		in detail in Section~\ref{sec:geieq}. However, we will
		focus our analysis on equilibrium prices that are free of arbitrage. Thus,
		we assume that $q$ and $p$ are no-arbitrage prices, rationalized by
		$\sigma$, and define an alternative vector of prices:
		\begin{equation}\label{eq:ptildetop}
			\tilde{p}_0=p_0,\qquad \tilde{p}_\xi=\sigma_\xi p_\xi,\quad\xi=1,\ldots,\Xi.
		\end{equation}
		Note that with this transformation $\tilde{p}_{\xi,0}$ is no longer equal to 1
		for $\xi=1,\ldots,\Xi$. In fact, $\tilde{p}_{\xi,0}$ corresponds to the associated state
		price $\sigma_\xi$.
		
		Additionally, let us introduce natural bounds on the consumption
		and contract spaces, given by
		\[c_\xi,w_\xi \leq m,\quad\xi=0,\ldots,\Xi,\qquad z^+,z^-\leq M(>0),\]
		where the first inequality follows from individual consumption being bounded by the
		aggregate available goods. Note that this set of inequalities also applies to the retention
		decision variables. Under these assumptions, and under the no-arbitrage condition,
		the consumption set turns out to be compact \cite[Ch.\,5,Thm.\,1]{Kir98ge}. Therefore,
		there is an upper bound on financial contracts, represented by the second inequality (note that this bound considers potential goods produced by home production).
		These considerations lead to the following formulation of the
		agent's problem, expressed in terms of the individual excess supply
		$s_\xi$ (defined in \eqref{sdef}) and the new price system $\tilde p$ as:
		\begin{equation}\label{modagprob}
			\begin{aligned}
				\max_{\{c,w,y,z^+,z^-\}}&\;U^i\left(c,w\right),\\
				{\rm s.t.}&\;\langle \tilde{p}_0,s_0\rangle-\sum_{\xi=1}^\Xi\langle \tilde{p}_\xi, D_\xi(z^+-z^-)\rangle \geq 0,\\
				&\langle \tilde{p}_\xi,s_\xi\rangle\geq 0,\quad\xi=1,\ldots,\Xi,\\
				&(c,w)\in C^i\cap m\ball_{(\Xi+1)(L+1)},\quad z^+,z^-\in M\ball_{J+1},\quad y\in Y^i.
			\end{aligned}
		\end{equation}
		A crucial observation arises from this substitution: the agent's
		problem now only depends on prices $\tilde p$ and is no longer
		dependent on prices $q$. Thus, defining the equilibrium
		conditions for this new version and solving the equilibrium for
		$\tilde p$ reduces the problem's dimension and its computational difficulty.
		
		\state Remark.
		Under Assumption~\ref{ass:indisp}, it is easy to see that Walras'
		law holds in this economy, and the budget constraints turn out
		to be active for every state $\xi=0,1,\ldots,\Xi$.
		
		Finally, the last element for establishing the equilibrium problem under this
		new setting is the recovery process for obtaining the corresponding
		prices of the original economy ($p,q$) from the prices of the
		modified problem, $\tilde p$. Thus, let $\tilde p$ be an equilibrium price for the modified
		economy (a precise definition of the equilibrium for the modified economy
		is given in Definition~\ref{def:eqmod}). First, we use Equation~\eqref{nacq} to obtain $q$ as follows:
		\begin{equation}\label{finq}
			q=\sum_{\xi=1}^\Xi D_\xi^\top \tilde{p}_\xi.
		\end{equation}
		
		From here, we compute the state prices (or multipliers) $\sigma$ that rationalize
		this system of prices by using the price normalization introduced in
		Equation~\eqref{eq:ptildetop}:
		\begin{equation}\label{mktp}
			p_{\xi,0}=1 \iff \sigma_\xi=\tilde{p}_{\xi,0}.
		\end{equation}
		
		Finally, the equilibrium price for commodities is determined by
		\[p_0=\tilde{p}_0,\quad p_{\xi,0}=1,\quad p_{\xi,l}=\frac{\tilde{p}_{\xi,l}}{\tilde{p}_{\xi,0}},\quad l=1,\ldots,L,\quad\xi=1,\ldots,\Xi.\]
		
		Note that by using this procedure, we characterize equilibrium prices
		for goods and financial contracts for the original problem, $p$ and $q$,
		by computing the equilibrium prices for commodities of the auxiliary
		economy, $\tilde p$. Specifically, the absence of financial contract prices in
		the auxiliary economy has further implications for the family of equilibrium
		problems that can be solved. As mentioned before, GEI models exhibit
		difficulties when the financial return matrix drops rank; computing equilibrium points then requires an extra procedure beyond classic homotopies. However, our approach overcomes
		this situation by considering prices on the \emph{manifold of no-arbitrage}. Thus, we compute prices for a modified pure-exchange economy where classic methods can be applied directly. Note that the concept of
		no-arbitrage equilibrium has been proposed before, for example by Hens in
		\cite[Ch.5,\S3.3]{Kir98ge}, but it has not been used for computational purposes
		to the best of our knowledge.
		
		The next step is to analyze the equilibrium conditions for the
		original economy and provide the optimization characterization
		of equilibrium points by defining the \emph{Walrasian bifunction}
		associated with this economy.
		
		\section{Equilibrium formulation and Approximation}\label{sec:geieq}
		In this section, we discuss the equilibrium conditions for this
		general equilibrium model. Consider an economy with normalized prices where the no-arbitrage condition has been incorporated into the formulation of the agents' utility maximization problems. That is, the agents' problems depend only on the alternative prices $\tilde p$
		(as in Problem~\ref{modagprob}).
		Let us denote this modified economy by $\Tilde \cE$ and define the domain of the price vector $\tilde p$ as the set $\Gamma:=\reals^{(1+L)(1+\Xi)-1}_+$
		(since $\tilde{p}_{0,0}=1$ is fixed). The equilibrium conditions are as follows:
		
		\begin{definition}{\rm (Equilibrium for the modified economy)}\label{def:eqmod}
			Consider the modified economy $\Tilde \cE$. An equilibrium
			for this economy is given by a system of prices
			${\tilde p}^*=({\tilde p}^*_0;{\tilde p}^*_1,\ldots,{\tilde p}^*_\Xi)\in\Gamma$
			and individual excess supply allocations $s^i=s^i({\tilde p}^*)$ for $i=1,\ldots,I$ (which depend on the optimal solutions $({c}^i,w^i,y^i,{z}^{+i},{z}^{-i})$ to the agents' problems) such that:
			\begin{enumerate}[label=(\alph*)]
				\item {\bf Utility maximization}: For every agent $i=1,\ldots,I$, the allocation solves the optimization problem in~\eqref{modagprob}; that is, it maximizes the utility function $U^i$ over the budget set
				\[\cB^i({\tilde p}^*)=\left\{(c,w,y,z^+,z^-)\left|\begin{array}{l}\displaystyle
					\langle {\tilde p}^*_0,s_0^i\rangle -\sum_{\xi=1}^\Xi \langle {\tilde p}^*_\xi,D_\xi (z^+-z^-)\rangle \geq 0,\\\displaystyle\langle {\tilde p}^*_\xi,s_\xi^i \rangle \geq 0,\quad\xi=1,\ldots,\Xi\\
					0\leq c,w\leq m,\quad y\in Y^i,\quad 0\leq z^+,z^-\leq M\end{array}\right.\right\}.\]
				
				\item {\bf Commodity Market Clearing}: The aggregate demand for
				goods at prices ${\tilde p}^*$ does not exceed the aggregate supply
				in the first stage, as well as in every possible state at the second stage; that is:
				\[{\rm ES}_{\xi,l}\geq 0,\quad l=0,\ldots,L,\quad\xi=0,\ldots,\Xi.\]
				
				\item {\bf Financial Market Clearing}: The market for financial contracts clears; that is:
				\[\sum_{i=1}^I { z}^{+,i}_j=\sum_{i=1}^I { z}^{-,i}_j,\quad j=0,\ldots,J.\]
			\end{enumerate}
		\end{definition}
		
		Under this definition of equilibrium, we propose a solution procedure
		based on solving an associated optimization problem,
		following the approach first proposed in \cite{JoWe02wlrs}. Here, equilibrium
		prices are characterized as maxinf-points of a so-called Walrasian bifunction, which
		captures the market imbalance in aggregate excess supply at given prices. This characterization is motivated by the existence proof in \cite{debreu1959theory}. Thus, the first step toward the maxinf characterization is to define an associated \emph{Walrasian function} for the modified economy. To this end, we define the dual variables $\tilde g\in\Gamma$ and the Walrasian bifunction $W:\Gamma\times\Gamma\to\reals$ as follows:
		\begin{equation}\label{walgei}
			\begin{aligned}
				W({\tilde p},{\tilde g})\,&=\langle {\tilde g},\Tilde {\rm ES}(\tilde p)\rangle-\rho\left\|\sum_{i=1}^I (z^{+i}- z^{-i})\right\|_2^2,\\ 
				&=\sum_{\xi=0}^\Xi\sum_{{\substack{l=0 \\ (\xi,l)\neq (0,0)}}}^L {\tilde g}_{\xi,l} {\rm ES}_{\xi,l}(\tilde p)-\rho\sum_{j=0}^J\left(\sum_{i=1}^I (z_j^{+i}- z_j^{-i})\right)^2,
			\end{aligned}
		\end{equation}
		for some positive parameter $\rho>0$, where $\Tilde {\rm ES}(\tilde p)$ corresponds to ${\rm ES}(\tilde p)$ without considering the first component (${\rm ES}_{0,0}(\tilde p)$).
		
		In what follows, we are interested in relating equilibria with maxinf-points of the
		Walrasian bifunction. A point $\tilde p$ is referred to as a \emph{maxinf-point} of $W$ if
		\[{\tilde p} \in\argmax_{p\in\Gamma} \left[\inf_{g\in \Gamma} W(p,g)\right],\]
		and, for $\varepsilon\geq0$, $\tilde p_\varepsilon$ is called an \emph{$\varepsilon$-maxinf-point} of $W$, if
		\[\inf_{g \in \Gamma} W(\tilde p_\varepsilon,g)\geq \sup_{p \in \Gamma} \inf_{g \in \Gamma} W(p,g) -\varepsilon.\] 
		The sets of such points are denoted by $\nargmaxinf W$ and
		$\varepsilon$-$\nargmaxinf W$, respectively.
		Notice that the definition of an $\varepsilon$-maxinf-point is based on an approximation of optimality for the function $W$. We will see later that
		this is the correct interpretation of an approximate maxinf-point in the context
		of an approximate equilibrium point.
		
		Let us state the first result about the relationship between
		maxinf-points of the Walrasian function and equilibrium points for the economy.

		\begin{lemma}[Walras equilibrium prices and maxinf-points]\label{eqmaxinf}
			Every maxinf-point ${\tilde p}\in \Gamma$ of the Walrasian bifunction $W$
			such that $\inf_{g \in \Gamma} W({\tilde p},g) \geq 0$ is an equilibrium point.
		\end{lemma}
		\begin{proof} 
			Let $\tilde p$ be a maxinf-point of $W$ such that $\inf_{g \in \Gamma} W({\tilde p},g)\geq 0$.
			This implies $W(\tilde{p}, g) \geq 0$ for all $g \in \Gamma$ if $W(\tilde{p}, \cdot)$ is, for example, lower semi-continuous and the infimum is attained, or by properties of infimum.
			Notice that, by Assumption~\ref{ass:indisp}, for any price $\tilde p$, the excess supply associated
			with the numéraire good (which is defined as the good indispensable to all agents) must be
			non-positive, i.e., ${\rm ES}_{0,0}(\tilde p)\leq 0$. Moreover, taking ${\tilde g}=0\in \Gamma$,
			\[0\leq W({\tilde p},0)= - \rho\sum_{j=0}^J\left(\sum_{i=1}^I (z_j^{+,i}(\tilde p)- z_j^{-,i}(\tilde p))\right)^2,\]
			which implies $\sum_{i=1}^I (z_j^{+,i}(\tilde p)- z_j^{-,i}(\tilde p))=0$ for $j=0,\ldots,J$. This entails the equilibrium conditions for financial contracts.
			Additionally, consider ${\tilde g}=\mathbf{e}_{\xi,l}\in \Gamma$ (the unit vector with 1 at position $(\xi,l)$ and 0 elsewhere, for $(\xi,l) \neq (0,0)$). Then,
			\[0\leq W({\tilde p},\mathbf{e}_{\xi,l}) = {\rm ES}_{\xi,l}(\tilde p) - \rho \sum_{j=0}^J\left(\sum_{i=1}^I (z_j^{+,i}(\tilde p)- z_j^{-,i}(\tilde p))\right)^2 = {\rm ES}_{\xi,l}(\tilde p).\]
			This provides the equilibrium conditions ${\rm ES}_{\xi,l}(\tilde p) \geq 0$ for those markets.
			Finally, the market clearing condition for the numéraire good, ${\rm ES}_{0,0}(\tilde p) \geq 0$, follows from Walras' law. Since we already have ${\rm ES}_{0,0}(\tilde p) \leq 0$, it must be that ${\rm ES}_{0,0}(\tilde p) = 0$.
			Thus, $\tilde p$ is an equilibrium point.
		\end{proof}
		
		Using Lemma~\ref{eqmaxinf}, we pose the problem of finding an equilibrium
		point as that of finding maxinf-points for the Walrasian function. We propose a
		method to solve this problem using the appropriate notion of convergence
		on the space of bifunctions, which corresponds to \emph{lopsided convergence}.
		Following the same approach that we initiated in \cite{DeJoWe16awlrs}, we
		build a family of approximating bifunctions that ancillary lop-converges to the
		Walrasian bifunction. Lopsided convergence (ancillary-tight) is aimed at the
		convergence of maxinf-points and, in our context, is defined as follows:
		
		\begin{definition}[Ancillary-tight lopsided convergence]\label{lopdef}
			A sequence $\lset W^\nu:\Gamma^\nu \times \Gamma^\nu \to \reals \rset_{\nu \in \nats}$ lop-converges ancillary-tight to the function $W:\Gamma\times \Gamma\to \reals$, denoted by $W^\nu\lopto W$, if
			\begin{description}
				\item[(a)] for all $g\in \Gamma$, and all sequences $(p^\nu\in \Gamma^\nu)$ with $p^\nu \to p\in \Gamma$, there exists a sequence $(g^\nu \in \Gamma^\nu)$ with $g^\nu \to g$ such that
				\[\limsup_\nu W^\nu(p^\nu,g^\nu) \leq W(p,g);\]
				\item[(b)] for all $p\in \Gamma$, there exists a sequence $(p^\nu \in \Gamma^\nu)$ with $p^\nu \to p$ such that for any sequence $(g^\nu \in \Gamma^\nu)$ with $g^\nu \to g$,
				\[\liminf_\nu W^\nu(p^\nu,g^\nu)\geq W(p,g)\quad \mbox{when } g\in \Gamma, \quad \mbox{and} \quad \liminf_\nu W^\nu(p^\nu,g^\nu)\to \infty\quad \mbox{when } g\notin \Gamma;\]
				\item[(b-t)] and for any $\varepsilon >0$ and any sequence $p^\nu \to p$, one can find a compact set $K_\varepsilon\subset \Gamma$ (possibly depending on $\lset p^\nu\rset$) such that for all $\nu$ sufficiently large,
				\[\inf_{g \in \Gamma^\nu\cap K_\varepsilon} W^\nu(p^\nu,g) \leq \inf_{g \in \Gamma^\nu} W^\nu(p^\nu,g)+\varepsilon.\]
			\end{description}
		\end{definition}
		
		We propose a family of approximating bifunctions, the \emph{augmented Walrasians}
		$\{W^\nu\}_{\nu\in\nats}$, based on a non-concave augmentation technique. This construction
		is an extension of our previous work \cite{DeJoWe16awlrs}. This family is
		defined for a non-decreasing sequence of scalars
		$\{r^\nu\}_{\nu\in\nats}$ such that $r^\nu\upto r$ (for a sufficiently large $r$), a non-decreasing sequence of compact sets $\{K^\nu\}_{\nu\in\nats}$ with $\Gamma^\nu = K^\nu \cap \Gamma$,
		and an augmenting function $\sigma:\Gamma\to\reals$. Consider the
		approximating domains $\Gamma^\nu = [0,K^\nu_{\text{vec}}]\cap \Gamma$ (if $K^\nu$ are vectors defining boxes), and define
		$W^\nu:\Gamma^\nu\times \Gamma^\nu\to\reals$ for $\nu\in\nats$ by:
		\begin{equation}\label{geiaw}
			W^\nu({\tilde p},{\tilde g})=\inf_z \Lset W({\tilde p},{\tilde g}-z)+r^\nu\sigma ^*\left(\frac{1}{r^{\nu}}z\right)\Rset.
		\end{equation}
		
		In what follows, we assume that the agents' utility maximization problem has a (locally) unique solution for every $\tilde p\in \Gamma$. Note that this condition is related to the classical assumption of negative definiteness of the Hessian of the utility functions over the entire domain. Additionally, this condition can be guaranteed by the assumption of strong convexity of preferences \cite[Thm.4.7]{rtrutility}. For further discussion on utility representation, consult \cite{rtrutility}.
		
		\begin{theorem}[Upper-semi continuity of the Walrasian]\label{theo:uscwalgei}
			The Walrasian function $W$ for the modified economy defined in \eqref{walgei} is upper semi-continuous (usc) in its first argument, i.e., for fixed ${\tilde g}\in \Gamma$, $W(\cdot, {\tilde g})$ is usc.
		\end{theorem}
		\begin{proof}
			For each agent $i$, define the choice variable $v_i=(c^i,w^i,y^i,z^{+,i},z^{-,i})$, and let $v=(v_1, \dots, v_I)$. Define the function
			\[f({\tilde p},v)=\sum_{i=1}^I U^i(c^i,w^i) \quad \text{if } v_i \in \cB^i({\tilde p}) \text{ for all } i, \text{ and financial markets clear},\]
			and $f({\tilde p},v) = -\infty$ otherwise. The agents' problems are to maximize $U^i$ subject to their individual budget constraints. The excess supply ${\rm ES}(\tilde p)$ depends on the optimal choices.
			Let $v^*(\tilde p)$ be the (assumed unique) collection of optimal choices for all agents.
			The function $P(\tilde p) = v^*(\tilde p)$ maps prices to optimal allocations.
			The individual budget sets $\mathcal{B}^i(\tilde p)$ are continuous multifunctions (correspondences) of $\tilde p$. Given that $U^i$ are usc and the budget sets are compact and vary continuously, Berge's Maximum Theorem implies that the optimal value function is continuous and the optimal solution correspondence $v_i^*(\tilde p)$ is upper hemi-continuous. If uniqueness is assumed, $v_i^*(\tilde p)$ becomes a continuous function.
			Then ${\rm ES}(\tilde p)$ is a continuous function of $\tilde p$. The term $\sum_{j=0}^J\left(\sum_{i=1}^I (z_j^{+i}- z_j^{-i})\right)^2$ is also continuous.
			Since $W(\tilde p, \tilde g)$ is a linear combination (for fixed $\tilde g$) of terms derived from continuous functions of $\tilde p$, $W(\cdot, \tilde g)$ is continuous, and therefore usc.
		\end{proof}
		
		\begin{lemma}[Ky-Fan properties of the augmented Walrasian]\label{lemm:kyfan}
			The finite-valued bifunctions $\{W^\nu\}_{\nu\in\nats}$, defined in \eqref{geiaw}, are Ky-Fan bifunctions. Moreover, for every $\nu\in\nats$, the set of maxinf-points of $W^\nu$ is nonempty.
		\end{lemma}
		\begin{proof}
			Recall that a bifunction $F:K\times K\to\reals$ is a Ky-Fan function if:
			\begin{itemize}[itemsep=-0.1cm]
				\item[i.] For every $y\in K$, the function $F(\cdot,y)$ is usc.
				\item[ii.] For every $x\in K$, the function $F(x,\cdot)$ is convex (or more generally, quasi-concave if we are looking for saddle points, but for maxinf, convexity in the minimization variable is key for the inner infimum). Here, $W^\nu(\tilde p, \cdot)$ needs to be convex for standard Ky-Fan results on existence of saddle points. The inf-projection preserves convexity if the original function is convex in that argument. $W(\tilde p, \tilde g - z)$ is linear in $(\tilde g - z)$, and $\sigma^*$ is convex. So $W^\nu(\tilde p, \cdot)$ is convex.
			\end{itemize}
			The augmented Walrasian $W^\nu(\tilde p, \tilde g)$ is the inf-projection of $W({\tilde p},{\tilde g}-z)+r^\nu\sigma^*(r^{-\nu}z)$ with respect to $z$.
			Condition (i): For a fixed $\tilde g \in \Gamma^\nu$, $W^\nu(\cdot, \tilde g)$ is usc. This follows from the usc of $W(\cdot, \cdot)$ (Theorem~\ref{theo:uscwalgei}) and properties of inf-projections (e.g., \cite[Prop. 1.26]{VaAn} if conditions apply, or more directly if $W$ is continuous in $\tilde p$).
			Condition (ii): For a fixed $\tilde p \in \Gamma^\nu$, $W^\nu(\tilde p, \cdot)$ is convex. The function $(\tilde g, z) \mapsto W(\tilde p, \tilde g - z) + r^\nu \sigma^*(r^{-\nu}z)$ is convex in $\tilde g$ (since $W$ is linear in its second argument and $\tilde g \mapsto \tilde g-z$ is affine) and convex in $z$ (if $\sigma^*$ is convex). The inf-projection of a convex function is convex \cite[Prop. 2.22]{VaAn}.
			Since $\Gamma^\nu$ is compact and non-empty, and $W^\nu$ is a Ky-Fan function on $\Gamma^\nu \times \Gamma^\nu$, the set $\nargmaxinf W^\nu$ is non-empty by Ky-Fan's inequality or related existence theorems for maxinf points (e.g., \cite[Thm. 2, Ch. 6]{aubin2006applied}).
		\end{proof}
		
		\begin{theorem}[Lop-convergence of augmented Walrasians]\label{geilopaw}
			The sequence of augmented Walrasian bifunctions $\{W^\nu\}$ converges lopsided ancillary-tightly to the
			original Walrasian $W$, i.e., $W^\nu \lopto W$. Moreover, every cluster point ${\tilde p}$ of any sequence
			$\{{\tilde p}^\nu\}$ of maxinf-points of $W^\nu$ is a maxinf-point of $W$.
		\end{theorem}
		\begin{proof}
			Using Definition~\ref{lopdef} of lopsided convergence:
			(a) Let ${\tilde g}\in \Gamma$, and let $({\tilde p}^\nu\in \Gamma^\nu)$ with ${\tilde p}^\nu\to {\tilde p}\in \Gamma$. Take ${\tilde g}^\nu={\tilde g}$ (assuming $\tilde g \in \Gamma^\nu$ for $\nu$ large enough, or choose $g^\nu \to g$). Then, from the definition of $W^\nu$, $W^\nu({\tilde p}^\nu,{\tilde g}^\nu) \leq W({\tilde p}^\nu,{\tilde g}^\nu)$ (by choosing $z=0$).
			\[\limsup_\nu W^\nu({\tilde p}^\nu,{\tilde g}^\nu) \leq \limsup_\nu W({\tilde p}^\nu,{\tilde g}^\nu).\]
			By the upper semi-continuity of $W$ in its first argument (Theorem~\ref{theo:uscwalgei}), $\limsup_\nu W({\tilde p}^\nu,{\tilde g}^\nu) \leq W({\tilde p},{\tilde g})$. Thus, condition (a) holds.
			
			(b) Let ${\tilde p}\in \Gamma$. Choose ${\tilde p}^\nu = {\tilde p}$ (assuming $\tilde p \in \Gamma^\nu$ or project it). Let $({\tilde g}^\nu \in \Gamma^\nu)$ with ${\tilde g}^\nu \to {\tilde g}$.
			The functions $h^\nu(z) = W({\tilde p},{\tilde g}^\nu-z)+r^\nu\sigma ^*(r^{-\nu}z)$ epi-converge to $h(z) = W({\tilde p},{\tilde g}-z)$ if $r^\nu \sigma^*(r^{-\nu}z)$ epi-converges to $0$ for $z \neq 0$ and $\infty$ for $z=0$ (this is not quite right, it should be $I_{\{0\}}(z)$ if $r^\nu \to \infty$). More standard is that $W^\nu(\tilde p, \cdot)$ epi-converges to $W(\tilde p, \cdot)$ as $\nu \to \infty$.
			This implies $\liminf_\nu \inf_z h^\nu(z) \geq \inf_z h(z)$, which means $\liminf_\nu W^\nu({\tilde p},{\tilde g}^\nu) \geq W({\tilde p},{\tilde g})$.
			The ancillary-tightness (b-t) follows because $\Gamma^\nu$ are expanding compact sets, and the infimum over $\Gamma^\nu \cap K_\varepsilon$ can be made close to the infimum over $\Gamma^\nu$ by choosing $K_\varepsilon$ appropriately if $W^\nu(p^\nu, \cdot)$ does not escape to $-\infty$ too fast outside $K_\varepsilon$.
			The convergence of maxinf-points follows from general results on lopsided convergence, e.g., \cite[Cor. 3.7]{JoWe14motivating}.
		\end{proof}
		
		\begin{corollary}[Ky-Fan properties of the Walrasian]
			The Walrasian bifunction $W$ is a Ky-Fan bifunction. Moreover,
			the set of maxinf-points of $W$ is nonempty.
		\end{corollary}
		\begin{proof}
			The bifunction $W$ is the lop-limit of the family of Ky-Fan functions $\{W^\nu\}_{\nu\in\nats}$ by Theorem~\ref{geilopaw}. By virtue of \cite[Thm. 5.2]{JoWe09lop}, $W$ is a Ky-Fan function (under suitable conditions, e.g., $\Gamma$ being compact, or $W$ having certain coercivity/boundedness properties). Therefore, $\nargmaxinf W\neq \emptyset$ by Ky-Fan's inequality or Sion's minimax theorem if $\Gamma$ is compact and $W$ satisfies convexity/concavity conditions.
		\end{proof}
		
		Note that the convergence result establishes the existence of maxinf-points of the Walrasian function, which represent the equilibrium conditions for the modified economy, as stipulated in Lemma~\ref{eqmaxinf}. The general conditions for the existence of equilibrium in the GEI model can be found in \cite{hart1975optimality}, but are beyond the scope of this article.
		
		Finally, we present a solution procedure for identifying GEI equilibrium points, as defined in Section~\ref{sec:geieq}. This procedure is rooted in the solution strategy for the agent's problem, as outlined in Section~\ref{sec:agprob}.
		
		We leverage the equivalence between GEI equilibrium and modified equilibrium, approximating the latter using Definition~\ref{def:eqmod} and applying Lemma~\ref{eqmaxinf} to represent equilibrium as maxinf-points of the associated Walrasian function $W$.
		
		To facilitate our approach, we introduce a family of approximating augmented Walrasian functions, denoted as $W^\nu$. We employ Theorem~\ref{geilopaw} to construct our algorithm, which entails computing a sequence of near-local saddle points (maxinf-points) of the family $\{W^\nu\}_{\nu\in\nats}$.
		
		\begin{itemize}
			\item At iteration $\nu+1$, given $(\tilde p^\nu,\tilde g^\nu)$, an augmenting parameter $r_{\nu+1} (\geq r_\nu)$, {\bf Phase I (or primal)} consists of solving
			\[\tilde g^{\nu+1} \in \argmin_{g\in \Gamma^\nu}W^{\nu+1}(\tilde p^\nu,g).\]
			Note that the 'internal' minimization is a convex optimization
			problem, which takes simpler forms depending on the selection
			of the augmenting function. In particular, it becomes the
			minimization problem of a linear form over a ball when $\sigma$
			is taken as a norm, or it could take the form of a quadratic
			minimization when $\sigma$ is the self-dual augmenting function.
			In these cases, the problem has an immediate solution.
			
			\item We define {\bf Phase II (or dual)} as finding
			\[\tilde p^{\nu+1} \in \argmax_{p\in \Gamma^\nu} W^{\nu+1}(p,\tilde g^{\nu+1}).\]
		\end{itemize}
		How to carry out this step will depend on the 'shape' and
		properties of the demand functions. For example, this turns out to
		be rather simple when the utility functions are of the Cobb-Douglas
		type, where smooth properties are transferred to the demand functions.
		
		By virtue of Theorem~\ref{geilopaw}, we know that as
		$r^\nu \upto r>0$ and $K^\nu\upto\infty$ (i.e., $\Gamma^\nu \to \Gamma$), $\tilde p^\nu \to \tilde p^*$,
		a maxinf-point of $W$, equivalently an equilibrium
		price system for the modified economy. The strategy for increasing
		$r^\nu$ should take into account (i) numerical stability, i.e., keep
		$r^\nu$ as small as possible, and (ii) efficiency, i.e., increase
		$r^\nu$ proportionately to the (reciprocal) of the imbalance of the market,
		to guarantee accelerated convergence.
		
		Finally, with the optimal prices for $\tilde\cE$, one can construct
		an equilibrium price for the original economy by following the
		computation described in \eqref{mktp} and \eqref{finq}.
		
		This new formulation highlights two major features. First, embedding the equilibrium problem into a family of perturbed optimization problems induces a notion of stability for each iteration, allowing the algorithm to perform its iterations robustly and without jumping too far from the current point. Secondly, the optimization nature of the maxinf problem opens a wide variety of well-known and
		developed computational libraries that can solve the corresponding
		optimization problem for each phase of the algorithm.
		
		\section{Numerical results}\label{geinum}
		
		\subsection*{Software implementation}
		
		The proposed algorithm was implemented using the Python programming language,
		along with the mathematical programming language Pyomo (Python Optimization Modeling Objects) \cite{pyomo,pyomo17}. The implemented solution addresses problems with the following features:
		
		Consider an initial price vector $\tilde p^0$ and the sequences
		of positive non-decreasing scalars $\{r^\nu\}$ and $\{K^\nu\}$ (defining $\Gamma^\nu$). Additionally,
		consider an augmenting function $\sigma$, which in this section is
		taken as $\sigma(x)=\frac{1}{2}|x|^2$, the self-dual function. Consider
		the sequence of associated augmented Walrasian functions, $\{W^\nu\}$,
		defined by \eqref{geiaw}. For a given tolerance level $\varepsilon$, the augmented
		Walrasian algorithm iterates as follows:
		\begin{description}
			\item[Step 0.] For a given price $\tilde p^\nu$, compute {\bf Phase I}, consisting of
			\[\tilde g^{\nu+1} \in \argmin_{g\in \Gamma^\nu} W^{\nu+1}(\tilde p^\nu,g),\]
			i.e., minimizing a quadratic objective function over the compact set $\Gamma^\nu$.
			This problem is solved using the Gurobi solver \cite{gurobi}, a state-of-the-art and efficient optimization software.
			
			\item[Step 1.] With the price $\tilde g^{\nu+1}$ obtained in Step 0,
			perform {\bf Phase II}:
			\[\tilde p^{\nu+1} \in \argmax_{p\in \Gamma^\nu} W^{\nu+1}(p,\tilde g^{\nu+1}).\]
			This is the critical step of the entire augmented Walrasian algorithmic
			framework. We need to overcome the (typical) lack of concavity of the
			objective function. Thus, the maximization is performed without considering
			first-order information, relying instead on the BOBYQA algorithm \cite{bobyqa}.
			BOBYQA performs a sequential local quadratic fit of the objective
			function over box constraints and solves it using a trust-region method.
			
			\item[Step 2.] If $\tilde p^{\nu+1}$ is not an $\varepsilon$-equilibrium for
			the economy $\tilde\cE$, then set $\tilde p^\nu\leftarrow \tilde p^{\nu+1}$,
			$r^\nu\leftarrow r^{\nu+1}$, $K^\nu \leftarrow K^{\nu+1}$ (update $\Gamma^\nu$), and go back
			to Step 0. Otherwise, stop: $\tilde p^{\nu+1}$ is an acceptable
			$\varepsilon$-equilibrium point for the modified economy. Then, compute $p_\varepsilon$
			and $q_\varepsilon$, the associated $\varepsilon$-equilibrium prices for the original
			equilibrium problem, following \eqref{finq} and \eqref{mktp}.
		\end{description}
		
		\state Remark.
		Notice that for each evaluation of the excess supply, it is necessary to
		solve the agent's maximization problem. For the examples covered in
		this section, the agents maximize a concave utility function over
		a linearly constrained set determined by budget constraints and
		non-negativity conditions. This problem is solved using the interior-point
		method implemented in Ipopt \cite{ipopt}, which yields satisfactory
		results for problems of this nature.
		
		All the examples were run on a Mac mini, 3 GHz Intel Core i5
		processor with 32 GB of RAM memory, running macOS Ventura 13.1.
		The corresponding routines and parameters are available in the official
		repository of the project\footnote{\href{https://github.com/jderide/GEI-awlrs}{\tt https://github.com/jderide/GEI-awlrs}},
		under GNU General Public License v3.0.
		
		\subsection{Brown, DeMarzo, Eaves example \cite{BrDeEa96gei}}\label{ssecbde}
		Our first example, presented in \cite{BrDeEa96gei}, is a two-stage stochastic exchange economy with
		an incomplete financial market (GEI), no retention, and no home production. The parameters are as follows: there are two types of agents, A and B, with twice as many B agents as A agents (i.e., $I=3$, representing agents A, B, B). There are two goods to be traded,
		$L=2$, and three possible states of nature for the second stage.
		
		The utility of each agent is given by the function
		\[U^i(c)=-\sum_{\xi=0}^\Xi \lambda_\xi^i \left(K-\prod_{l=1}^2 c_{l,\xi}^{\alpha_{i,l}}\right)^2,\]
		where $\xi=0$ represents the first stage, and $\lambda$ is a factor incorporating the
		probability distribution (in this example, $\lambda_{\xi}^i=(1,1/3,1/3,1/3)$ for every agent $i$ and every second-stage scenario $\xi$).
		Additionally, consider $K=5.7$, and the agent's parameters $\alpha$ and $e$ given by:
		\begin{align*}
			\alpha_A=(1/4,3/4),&\quad \alpha_B=(3/4,1/4),\\
			e_A=\left(\begin{array}{cccc}2&0.5&1&1.5\\
				2&1&1&1\end{array}\right),&\quad e_B=\left(\begin{array}{cccc}1&2.5&2&1.5\\
				1&2&2&2\end{array}\right).
		\end{align*}
		
		The financial market consists of two securities with the return matrix $D$ given by
		\[D_\xi=\left[\begin{array}{cc}1&0\\0&1\end{array}\right],\quad\xi=1,2,3.\]
		
		Note that for this market of securities, the no-arbitrage condition translates to
		the existence of positive multipliers $\sigma_\xi>0$ for $\xi=1,2,3$ such that
		\[\binom{q_0}{q_1}=\binom{\sum_{\xi=1}^3 \sigma_\xi p_{\xi,0}}{\sum_{\xi=1}^3 \sigma_\xi p_{\xi,1}} = \binom{\sum_{\xi=1}^3 \sigma_\xi \cdot 1}{\sum_{\xi=1}^3 \sigma_\xi p_{\xi,1}}.\]
		(Assuming $D_{\xi,0j}$ corresponds to $p_{\xi,0}$ and $D_{\xi,1j}$ to $p_{\xi,1}$ in the sum $\sigma_\xi D_\xi^\top p_\xi$. The original expression was $\sigma_1p_1(1)+\sigma_2p_1(2)+\sigma_3p_1(3)$ which is less general. If $p_{\xi,0}=1$, then $q_0 = \sum \sigma_\xi$. If the second security pays out good 1, then $q_1 = \sum \sigma_\xi p_{\xi,1}$.)
		
		An interesting feature of this economy arises from the structure of the
		first financial contract: since good 0 is the numéraire, this contract
		delivers one unit of this good in every possible state. Thus, it can be interpreted
		as a bond, and its price $q_0$ corresponds to its
		expected payoff (discounted by state prices), i.e., $q_0 = \sum_{\xi=1}^3 \sigma_\xi$. The interest rate for this economy can be computed as
		$r=(1/q_0)-1$.
		
		\subsubsection{Numerical Solution}
		
		Applying the augmented Walrasian algorithm to the
		auxiliary Walrasian function $\tilde W$ yielded an equilibrium point
		(to $10^{-2}$ precision). The initialization of the algorithm played an
		important role in the speed of convergence. The system
		of equilibrium prices for the auxiliary problem is given by
		\[{\tilde p}^*=\left[\begin{array}{c|ccc}
			1.00000&0.28871&0.31709&0.32718\\
			0.75482&0.22222&0.22864&0.21409
		\end{array}\right].\]
		
		From here, the equilibrium prices for the original economy are given by
		\[p^*=\left[\begin{array}{c|ccc}
			1.00000&1.00000&1.00000&1.00000\\
			0.75482&0.7697&0.7210&0.6544
		\end{array}\right],\quad
		q^*=\binom{0.932986}{0.664952},\]
		and using the calculation described previously, the equilibrium interest rate is ${\bf r} \approx 7.18\%$.
		
		Figure~\ref{fig:itbde} depicts the iterations performed by the algorithm.
		
		\begin{figure}[!hbt]
			\begin{center}
				\includegraphics[width=0.8\textwidth]{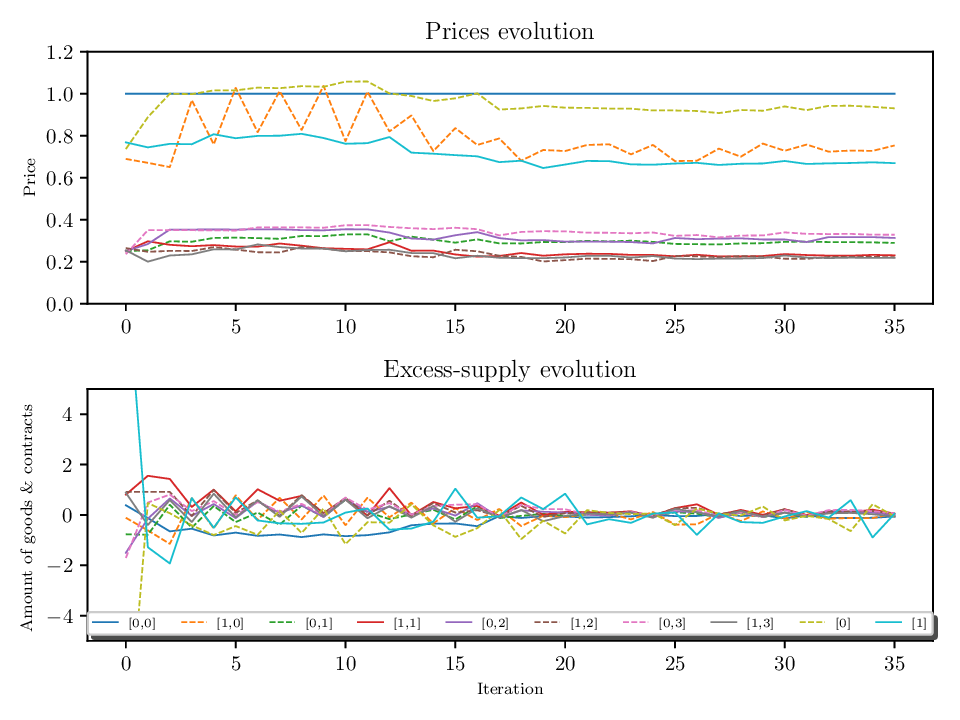}
				\caption{\small Augmented Walrasian iterations, evolution of prices ($\tilde p$, top) and Excess supply (${\rm ES}$, bottom)}
				\label{fig:itbde}
			\end{center}
		\end{figure}
		
		\subsubsection{Comparison of Numerical Results}
		In this section, we study the numerical results for the BDE
		example, comparing the original implementation in \cite{BrDeEa96gei},
		the modified implementation by \cite{Ma15compgei}, and our approach.
		
		Note first that allocations $c$, contracts $z$, and contract prices $q$ are
		derived from normalizing the prices $p$ reported in both papers (setting the price of the first good to 1) and solving the agent's problem with
		our implementation. A summary of the different equilibria is described
		in Table~\ref{tab:comp}.
		
		\begin{table}[h!]
			\begin{center}
				\begin{scriptsize}
					\begin{tabular}{ccccc|cccc|cccc}
						\toprule
						& \multicolumn{4}{c}{\textbf{Ma}\cite[Ex1]{Ma15compgei}} & \multicolumn{4}{c}{\textbf{BDE}\cite[5.1]{BrDeEa96gei}} & \multicolumn{4}{c}{\textbf{DW}} \\
						\textbf{c{[}i{]}{[}l,s{]}} & \textbf{Ag{[}0{]}} & \textbf{Ag{[}1,2{]}} & \textbf{ES} & \textbf{p} & \textbf{Ag{[}0{]}} & \textbf{Ag{[}1,2{]}} & \textbf{ES} & \textbf{p} & \textbf{Ag{[}0{]}} & \textbf{Ag{[}1,2{]}} & \textbf{ES} & \textbf{p} \\
						\midrule
						{[}0,0{]} & 0.600 & 1.690 & 0.019 & 1.0000 & 0.603 & 1.699 & -0.001 & 1.0000 & 0.605 & 1.726 & -0.058 & 1.0000 \\
						{[}1,0{]} & 2.464 & 0.771 & -0.006 & 0.7307 & 2.451 & 0.768 & 0.013 & 0.7375 & 2.406 & 0.762 & 0.069 & 0.7548 \\
						{[}0,1{]} & 0.766 & 2.369 & -0.003 & 0.2885 & 0.773 & 2.367 & -0.007 & 0.2890 & 0.758 & 2.365 & 0.011 & 0.2887 \\
						{[}1,1{]} & 2.981 & 1.024 & -0.030 & 0.2224 & 2.965 & 1.010 & 0.016 & 0.2259 & 2.956 & 1.024 & -0.004 & 0.2222 \\
						{[}0,2{]} & 0.724 & 2.135 & 0.006 & 0.3057 & 0.716 & 2.149 & -0.015 & 0.3056 & 0.744 & 2.109 & 0.038 & 0.3171 \\
						{[}1,2{]} & 3.039 & 0.996 & -0.031 & 0.2184 & 2.994 & 0.999 & 0.008 & 0.2193 & 3.094 & 0.975 & -0.044 & 0.2286 \\
						{[}0,3{]} & 0.680 & 1.903 & 0.014 & 0.3252 & 0.688 & 1.907 & -0.002 & 0.3256 & 0.677 & 1.897 & 0.028 & 0.3272 \\
						{[}1,3{]} & 3.100 & 0.965 & -0.030 & 0.2138 & 3.112 & 0.958 & -0.029 & 0.2159 & 3.105 & 0.966 & -0.038 & 0.2141 \\
						\midrule
						\textbf{z{[}i{]}{[}j{]}} & \textbf{Ag{[}0{]}} & \textbf{Ag{[}1,2{]}} & \textbf{EC} & \textbf{q} & \textbf{Ag{[}0{]}} & \textbf{Ag{[}1,2{]}} & \textbf{EC} & \textbf{q} & \textbf{Ag{[}0{]}} & \textbf{Ag{[}1,2{]}} & \textbf{EC} & \textbf{q} \\
						{[}0{]} & -6.588 & 3.239 & -0.110 & 0.9194 & -6.247 & 3.210 & 0.173 & 0.9203 & -6.304 & 3.162 & 0.020 & 0.9330 \\
						{[}1{]} & 10.872 & -5.348 & 0.177 & 0.6547 & 10.306 & -5.267 & -0.228 & 0.6611 & 10.481 & -5.259 & -0.037 & 0.6650\\
						\bottomrule
					\end{tabular}
				\end{scriptsize}
			\end{center}
			\caption{Comparison: Equilibrium allocations, financial contracts, excess supply, and prices}
			\label{tab:comp}
		\end{table}
		
		Rigorous metrics for comparing the application of these algorithms
		are beyond the scope of our exposition. Nevertheless, our implementation of
		the augmented Walrasian algorithm for this equilibrium problem performs
		robustly for several starting points. Moreover, as the underlying structure
		of the equilibrium problem solved incorporates the no-arbitrage condition,
		each iteration depicted in Figure~\ref{fig:itbde} provides a notion of stability.
		Additionally, no special attention needs to be paid to intermediate steps
		where the market becomes incomplete or the space spanned by the financial
		contracts at a particular price loses rank.
		
		\subsection{Two modifications to BDE example (Example~\ref{ssecbde}) by Schmedders \cite{Sch98compGEI}}
		Following Schmedders' modifications to the previous example,
		consider the following problems:
		\subsubsection{Endowments variation \cite[Ex.2,\S6.2]{Sch98compGEI}}\label{ssecsch1}
		The first modification to Example~\ref{ssecbde} considers the same
		economy as before. Besides a substitution in the initial endowment of agent A,
		\[e_A=\left(\begin{array}{cccc}2.0&0.8&1.0&0.6\\ 
			2.0&2.4&3.0&1.8\end{array}\right),\]
		the other parameters remain the same. Running these experiments, we found the
		equilibrium prices for the auxiliary problem are:
		\[{\tilde p}^*=\left(\begin{array}{cccc}1.0000&0.3230&0.2769&0.3928\\ 
			0.8013&0.2050&0.1640&0.1969\end{array}\right),\]
		and using the transformation in \eqref{mktp}, the equilibrium prices for the original
		economy are:
		\[p^*=\left(\begin{array}{cccc}1.0000&1.0000&1.0000&1.0000\\ 
			0.8013&0.6347&0.5921&0.5013\end{array}\right),\quad q^*=\binom{0.9928}{0.5659},\quad r^*\approx 0.73\%.\]
		Figure~\ref{fig:Sch1} depicts the iterations of the algorithm; our method avoids the numerical instabilities that naturally arise from homotopy-type algorithms.
		\begin{figure}[!hbt]
			\begin{center}
				\includegraphics[width=0.8\textwidth]{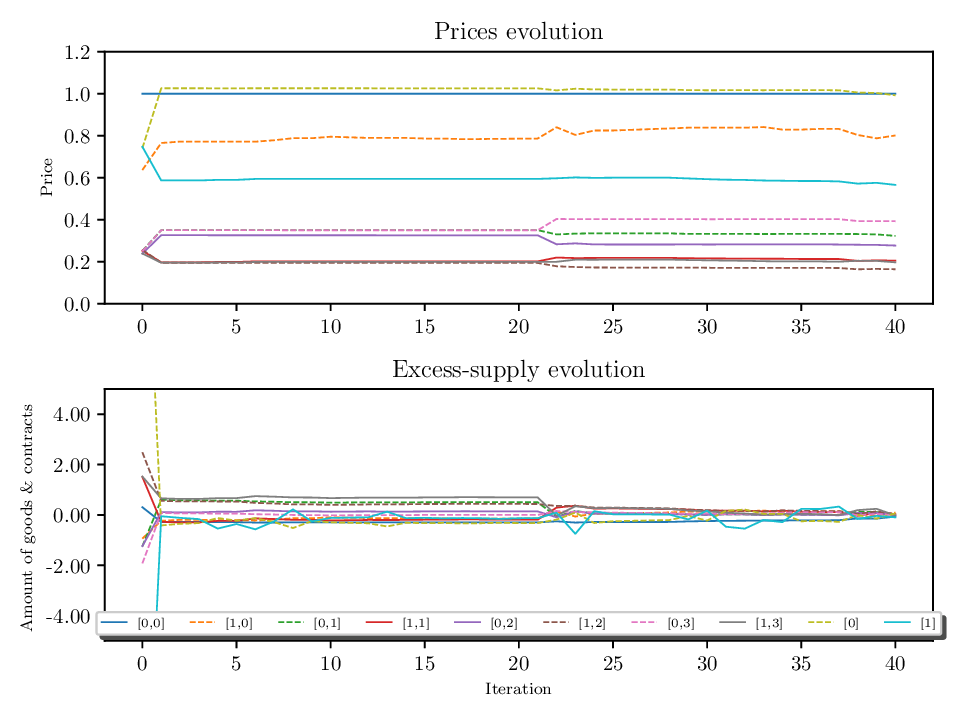}
				\caption{\small Example~\ref{ssecsch1}, evolution of prices ($\tilde p$, top) and Excess supply (${\rm ES}$, bottom)}
				\label{fig:Sch1}
			\end{center}
		\end{figure}
		
		\subsubsection{Homogeneous agents, no financial trade \cite[Ex.3,\S6.2]{Sch98compGEI}}\label{ssecsch2}
		This variation of Example~\ref{ssecbde} considers only two agents,
		where the special feature is the constant
		total initial endowment $\sum_{i=1}^I e^i_{l,\xi}=3$ for $\xi=0,\ldots,\Xi$ and $l=0,1$.
		By symmetry, this economy should have equilibrium prices that are identical.
		Our results are similar to those found in the original paper, where these
		optimal prices are called 'bad prices'. Figure~\ref{fig:Sch3}
		shows the behavior of the augmented Walrasian; it is easy to see the asymptotic
		convergence to the equilibrium prices.
		
		\begin{figure}[!hbt]
			\begin{center}
				\includegraphics[width=0.8\textwidth]{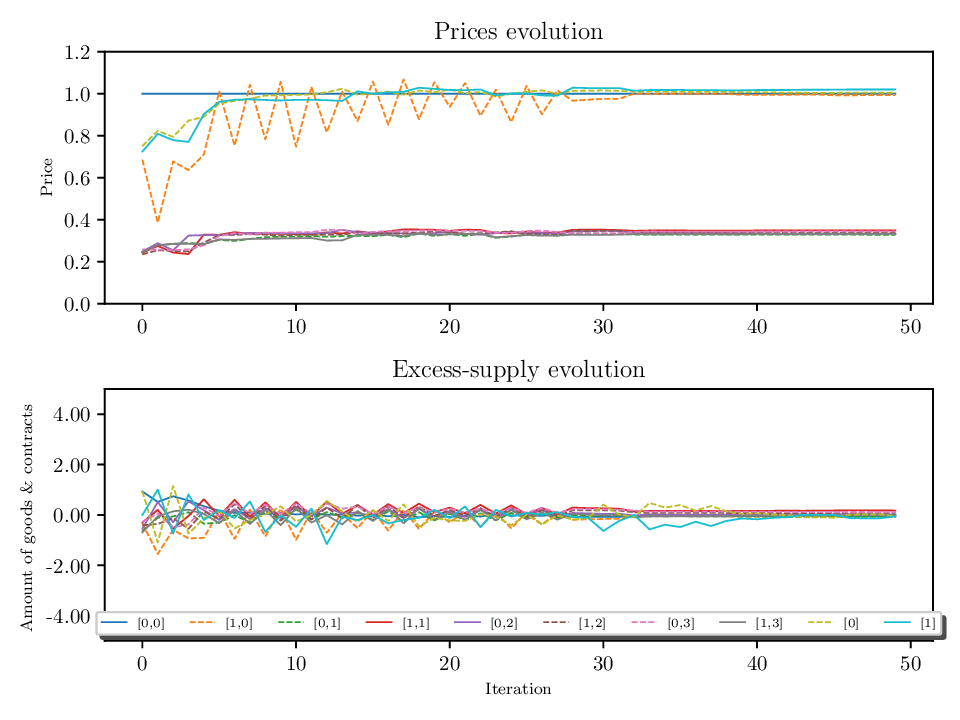}
				\caption{\small Example~\ref{ssecsch2}, evolution of prices ($\tilde p$, top) and Excess supply (${\rm ES}$, bottom)}
				\label{fig:Sch3}
			\end{center}
		\end{figure}
		
		\subsection{Heterogeneous agents, retention, and financial trade}\label{ExBig}
		In this example, we consider a larger economy: $5$ agents, $3$ possible second-stage scenarios, $6$ commodities, and $5$ financial contracts. Agents' utility functions are described by
		\[U^i(c,w)=-\sum_{\xi=0}^\Xi \lambda_\xi^i \left(K-\prod_{l=1}^L c_{l,\xi}^{\alpha^i_{l,\xi}} - \rho\prod_{l=1}^L w_{l,\xi}^{\alpha^i_{l,\xi}}\right)^2,\]
		where $\xi=0$ represents the first stage, and $\lambda$ is a factor incorporating the
		probability distribution (in this example, $\lambda_{\xi}^i=(1,1/3,1/3,1/3)$ for all $i, \xi$).
		Additionally, consider the following agent parameters: $K=5.7$, $\rho=0.01$. The utility coefficients are $\alpha^i_{l,\xi}=\alpha^i_l$ for $\xi=0,\ldots,3$ and are given in Table~\ref{tab:ExBigalpha}.
		\begin{table}[ht]
			\begin{center}
				\begin{tabular}{c|cccccc}
					\toprule
					$\alpha^i_l$&$l=1$&$l=2$&$l=3$&$l=4$&$l=5$&$l=6$\\
					\midrule
					$i=1$&0.25&0.25&0.17&0.17&0.08&0.08\\
					$i=2$&0.25&0.25&0.17&0.17&0.08&0.08\\
					$i=3$&0.17&0.17&0.25&0.25&0.08&0.08\\
					$i=4$&0.17&0.17&0.25&0.25&0.08&0.08\\
					$i=5$&0.08&0.08&0.08&0.08&0.5&0.17\\
					\bottomrule
				\end{tabular}
			\end{center}
			\caption{Coefficients for utility functions in Example~\ref{ExBig}.}
			\label{tab:ExBigalpha}
		\end{table}
		
		The initial endowments $e^i_{l,\xi}$ are taken as variations of a base vector $e^i_{l}$ with a stage-dependent random component, following this rule:
		\[e^i_{l,0}=e^i_{l},\quad \begin{cases} e^i_{l,1}=e^i_{l}+\varepsilon\\ e^i_{l,2}=e^i_{l}\\e^i_{l,3}=e^i_{l}-\varepsilon\end{cases},\]
		where $\varepsilon\sim{\rm U}(0,0.1)$, a uniformly distributed random variable supported in $[0,0.1]$, and the base vector is given by Table~\ref{tab:ExBige}.
		
		\begin{table}[ht]
			\begin{center}
				\begin{tabular}{c|cccccc}
					\toprule
					$e^i_l$&$l=1$&$l=2$&$l=3$&$l=4$&$l=5$&$l=6$\\
					\midrule
					$i=1$&1&4&2&2&3&1.5\\
					$i=2$&4&1&2&2&1&1.5\\
					$i=3$&2&1&4&3&1&1.5\\
					$i=4$&2&4&1&3&1&1.5\\
					$i=5$&2&2&2&2&2&2\\
					\bottomrule
				\end{tabular}
			\end{center}
			\caption{Base vector defining the initial endowments in Example~\ref{ExBig}.}
			\label{tab:ExBige}
		\end{table}
		
		In addition, the retention-related matrices are identity matrices, $\{A_\xi=\mathbf{I}:\xi=1,\ldots,\Xi\}$. The financial market is formed by forward contracts, given by
		\[D_0=\mathbf{0}_{6\times 5},\quad D_\xi=\left[\begin{array}{c}{\mathbf I}_5\\ \mathbf{0}_{1\times 5}\end{array}\right],\quad \xi=1,\ldots,\Xi,\]
		i.e., no transaction costs are incurred when issuing a contract, and each contract
		represents a return of one unit for each of the first $J$ goods in every possible second-stage scenario.
		Figure~\ref{fig:Ex4} depicts the trajectories of the prices (top) and excess supply (bottom) for each iteration of the augmented Walrasian algorithm. Note that this problem has a size of
		\[\underbrace{I}_{\rm Agents}\cdot \left(\underbrace{2\cdot (L+1)\cdot (\Xi+1)}_{\rm Consumption+retention}+\underbrace{(J+1)}_{\rm Contracts}\right)+\underbrace{(L+1)(\Xi+1)}_{\rm Prices-goods}+\underbrace{(J+1)}_{\rm Prices-contracts}=294\]
		variables, considering the decisions made by each agent (consumption, retention, and contracts at each stage) and the total number of prices to be determined.
		\begin{figure}[!hbt]
			\begin{center}
				\includegraphics[width=0.8\textwidth]{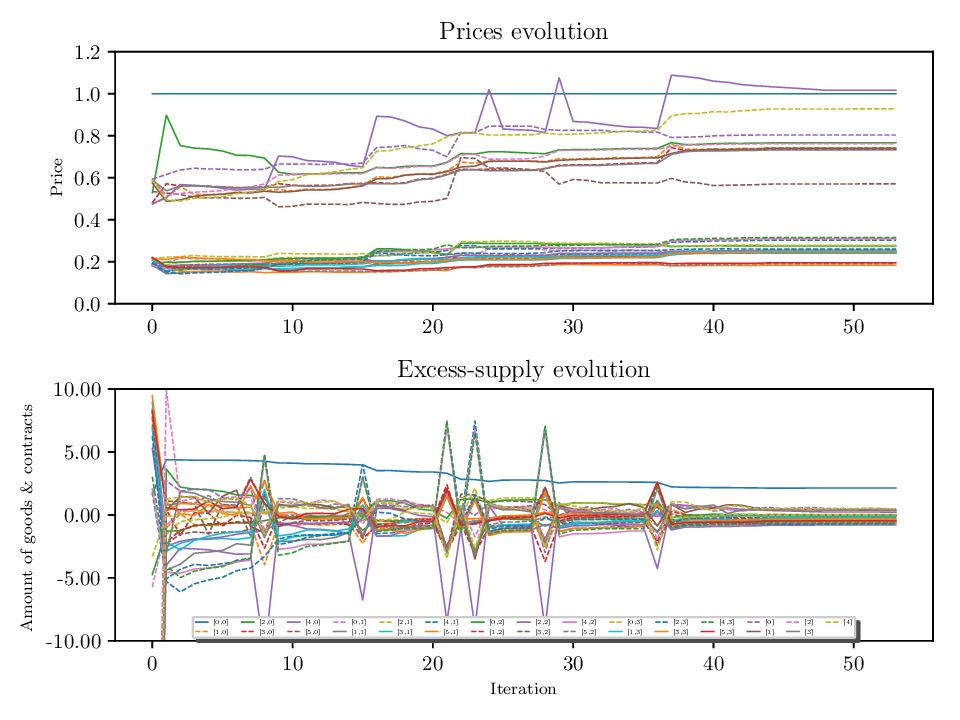}
				\caption{\small Example~\ref{ExBig}, evolution of prices ($\tilde p$, top) and Excess supply (${\rm ES}$, bottom)}
				\label{fig:Ex4}
			\end{center}
		\end{figure}
		
		The algorithm stopped after 54 iterations with a tolerance of $\varepsilon=0.1$, yielding the resulting price vectors:
		\[p^*=\left[\begin{array}{cccc}
			1&1&1&1\\
			0.7421&0.9681&0.9003&0.8995\\
			0.7666&0.9913&0.9261&0.9334\\
			0.7355&0.9585&0.8947&0.8856\\
			1.0166&1.2086&1.1287&1.1326\\
			0.5706&0.7317&0.6928&0.7007
		\end{array}\right],\quad q^*=\left[\begin{array}{c}
			0.8035\\
			0.7401\\
			0.7625\\
			0.7323\\
			0.9280
		\end{array}\right].\]
		
		\state Acknowledgement. I am deeply grateful to the late Roger J-B Wets for his invaluable contributions to this article. His expertise and insightful suggestions were instrumental in shaping this work, and I feel privileged to have had the opportunity to learn from him. Though he is no longer with us, his intellectual legacy continues to inspire my research.

\bibliographystyle{abbrv}
\bibliography{Deride-GEI}

\end{document}